\documentclass[a4paper,10pt]{article}

\usepackage[UKenglish]{babel}
\usepackage{amsmath,amscd,amsthm,amssymb}

\theoremstyle{plain}
\newtheorem{theorem}{Theorem}[section]
\newtheorem{lemma}[theorem]{Lemma}
\newtheorem{proposition}[theorem]{Proposition}

\theoremstyle{definition}
\newtheorem{definition}[theorem]{Definition}

\theoremstyle{remark}
\newtheorem{remark}[theorem]{Remark}

\begin{document}

\title{On the probability of satisfying a word in nilpotent groups of class $2$}
\author{Matthew Levy\\\\Imperial College London}
\date{}
\maketitle

\begin{abstract}
Let $G$ be a finite group of nilpotency class $2$ and $w$ a group word. In this short paper we show that the probability that a random $n$-tuple of elements from $G$ satisfies $w$ is at least one over the order of $G$. This answers a special case of a conjecture of Alon Amit.
\end{abstract}


\section{Introduction}
Let $G$ be a finite group, $w(x_{1},...,x_{n})$ a group word and denote by $N(G,w=c)$ the number of $n$-tuples $\textbf{g}=(g_{1},...,g_{n})\in G^{(n)}$ satisfying $w(\textbf{g})=c$, that is
$$
N(G,w=c)=|\{\textbf{g}\in G^{(n)}:w(\textbf{g})=c\}|.
$$
Also, denote by $P(G,w=c)$ the probability that a random $n$-tuple $\textbf{g}=(g_{1},...,g_{n})\in G^{(n)}$ satisfies $w(\textbf{g})=c$, that is
$$
P(G,w=c)=\frac{N(G,w=c)}{|G|^{n}}.
$$
When $c=1$ we will just write $N(G,w)$ and $P(G,w)$ and we will say that $\textbf{g}$ \textit{satisfies} $w$ if $w(\textbf{g})=1$. If $G$ were abelian, then the word map
$$
w:G^{(n)}\rightarrow G,
$$ 
defined by a word $w$ is a homomorphism and it is clear that 
$$
N(G,w)=|\mbox{Ker } w|=\frac{|G|^{n}}{|\mbox{Im }w|}\geq |G|^{n-1}
$$ 
and so $P(G,w)\geq\frac{1}{|G|}$.\\\\
It is a conjecture of Alon Amit (see $[\mathbf{1}]$ or $[\mathbf{2}]$) that if $G$ is a nilpotent group then $P(G,w)\geq\frac{1}{|G|}$. Here we establish the result for nilpotency class $2$ groups in the following theorem.

\begin{theorem}
Let $G$ be a finite group of nilpotency class $2$. Then for any group word $w$, $P(G,w)\geq\frac{1}{|G|}$.
\end{theorem}
This result improves the bound in the nilpotency class $2$ case established in a paper by Nikolov and Segal (see $[\mathbf{3}]$).
\begin{remark}
If the theorem holds true for two groups $G_{1}$ and $G_{2}$ then it holds true for their direct product $G=G_{1}\times G_{2}$ since the word values can be solved componentwise. Let $\textbf{g}=(g_{1},...,g_{n})$ be an $n$-tuple in $G$ then $\textbf{g}=\textbf{a}.\textbf{h}$ where $\textbf{a}=(a_{1},...,a_{n})$ and $\textbf{h}=(h_{1},...,h_{n})$ are $n$-tuples in $G_{1}$ and $G_{2}$ respectively. Then it is clear that $w(\textbf{g})=w(\textbf{a}).w(\textbf{h})$. Hence $P(G,w)\geq P(G_{1},w).P(G_{2},w)$ and the result follows since $|G|=|G_{1}|.|G_{2}|$. In fact, if $A$ is abelian and the theorem holds for a group $H$ which acts on $A$ by automorphisms then the theorem holds for $G=A\rtimes H$.

\end{remark}

\begin{proposition}
Let $G$ be a finite group such that $G=A\rtimes H$ where $A$ and $H$ are subgroups of $G$ and $A$ is abelian. Suppose that $P(H,w)\geq\frac{1}{|H|}$ where $w$ is a group word then $P(G,w)\geq\frac{1}{|G|}$.
\end{proposition}

\begin{proof}
For any $g\in G$ we may write $g=ah$ for unique $a\in A$ and $h\in H$, so if $w$ is a word in $n$ variables we have, for some $a_{i}\in A$ and $h_{i}\in H$,
\begin{eqnarray*}
w(g_{1},...g_{n}) &=& w(a_{1}h_{1},...,a_{n}h_{n})\\
                  &=& \displaystyle\prod_{i=1}^{n}a_{i}^{\phi_{i}(h_{1},...,h_{n})}w(h_{1},...,h_{n}),
\end{eqnarray*}
where the $\phi_{i}(h_{1},...,h_{n})$ are automorphisms of $A$ depending on the $h_{i}$. Note that there are at least $|H|^{n-1}$ $n$-tuples, $\textbf{h}=(h_{1},...,h_{n})\in H^{(n)}$, satisfying $w$. Having fixed such an $\textbf{h}\in H^{(n)}$ consider the induced map on $A^{(n)}$,
\begin{eqnarray*}
T_{\textbf{h}}:A^{(n)} &\longrightarrow& A\\
(a_{1},...,a_{n}) &\longmapsto& \displaystyle\prod_{i=1}^{n}a_{i}^{\phi_{i}\textbf(h)}.
\end{eqnarray*}
Since $A$ is an abelian group, $T_{\textbf{h}}$ is a linear map and the number of $n$-tuples $\textbf{a}\in A^{(n)}$ such that $T_{\textbf{h}}(\textbf{a})=1$ is at least $|A|^{n-1}$. The result follows.
\end{proof}

Since any nilpotent group is a direct product of its Sylow subgroups, by the above remark it will be enough to prove the following theorem:

\begin{theorem}
Let $G$ be a finite $p$-group of nilpotency class 2 where $p$ is a prime. Then for any group word $w$, $P(G,w)\geq\frac{1}{|G|}$.
\end{theorem}

\section{Nilpotent class $2$ groups}

We begin by making the following definition:
\begin{definition}
\label{DefnEquiv}
We will say that two group words $w_{1}$ and $w_{2}$ on $n$ variables are \textit{$G$-equivalent} if $N(G,w_{1}=c)=N(G,w_{2}=c)$ for every $c\in G$.
\end{definition}

\begin{remark}
\label{remark2}
It is clear that relabelling the variables of $w$ gives \textit{$G$-equivalent} words for any group $G$ since the word maps are unchanged. Suppose that $w$ and $w'$ are two group words in $n$ variables such that $w\equiv w'\hspace{0.5 mm}\mbox{mod }R$ where $R$ is a normal subgroup of the free group of rank $n$, $F_{n}$. Then it is easy to see that $w$ and $w'$ are \textit{$G$-equivalent} for any group $G$ that is a homomorphic image of $F_{n}/R$ since $v(G)=1$ for any $v\in R$. Also, suppose $w(x_{1},...,x_{n})$ is a group word and that $w(x_{1},...,x_{n})=v(y_{1},...,y_{n})$ where $v$ is a group word and the $y_{i}$ are words in the $x_{i}$s. Let $\mathfrak{N}_{2,m}$ denote the class of all finite $p$-groups of nilpotency class at most $2$ and of exponent at most $p^m$. If the set $\{y_{1},...,y_{n}\}$ maps onto a basis for the vector space $L_{n}/\Phi(L_{n})$ where $L_{n}$ is the free $\mathfrak{N}_{2,m}$-group on $n$ generators then $w$ is \textit{$G$-equivalent} to the word $v(y_{1},...,y_{n})$ for any $G\in\mathfrak{N}_{2,m}$, since for each $i$ the image of $x_{i}$ in $L_{n}$ can be expressed as a word in the $y_{j}$s. This follows from the Burnside Basis Theorem (see [$\mathbf{4}$]).
\end{remark}
The following commutator identities will be used throughout and are easy to prove: 
\newcounter{counter}
\begin{list}{(\arabic{counter})}{\usecounter{counter}}
\item $[x,y]=x^{-1}x^{y}$;
\item $[xy,z]\equiv[x,z][y,z]\hspace{5 mm}\mbox{mod }\gamma_{3}(F_{n})$;
\item $[x,yz]\equiv[x,y][x,z]\hspace{5 mm}\mbox{mod }\gamma_{3}(F_{n})$;
\item $[x,y]^{-1}\equiv[y,x]\equiv[x^{-1},y]\equiv[x,y^{-1}]\hspace{5 mm}\mbox{mod }\gamma_{3}(F_{n})$;
\item $(xy)^{n}\equiv x^{n}y^{n}[y,x]^{\frac{n(n-1)}{2}}\hspace{5 mm}\mbox{mod }\gamma_{3}(F_{n})$.
\end{list}
Let $w(x_{1},...,x_{n})$ be any group word and fix a group $G\in\mathfrak{N}_{2,m}$. By Hall's Collecting Process (see [$\mathbf{4}$]), we can write $w$ in the form
\begin{eqnarray}
w(x_{1},...,x_{n})=x_{1}^{\alpha_{1}}...x_{n}^{\alpha_{n}}(\displaystyle\prod_{i=1}^{n}\prod_{i<j}[x_{i},x_{j}]^{\beta_{i j}})c,
\end{eqnarray}
where $c\in \gamma_{3}(F_{n})$, $F_{n}$ being the free group of rank $n$ and $\alpha_{i}$, $\beta_{i j}\in\mathbb{Z}$.  We aim to show that the word map given by $w$ is `\textit{equivalent over $G$}', in the sense of definition \ref{DefnEquiv}, to the word map given by a particular word $w'$, where it will be easy to see that $P(G,w')\geq\frac{1}{|G|}$. To do this we will first prove the following lemmas.
\begin{lemma}
\label{LemmaVariablePower}
Let $w(x_{1},...,x_{n})=x_{1}^{\alpha_{1}}...x_{n}^{\alpha_{n}}$ be an element of the free group of rank $n$ and $m\in \mathbb{N}$. 
Then, for any $G\in\mathfrak{N}_{2,m}$, $w$ is \textit{$G$-equivalent} to the word
$$
v(y_{1},x_{2},...,x_{n})=y_{1}^{p^{l}}\displaystyle\prod_{1<i<j}[x_{i},x_{j}]^{\beta_{ij}}\prod_{i=2}^{n}[y_{1},x_{i}]^{\gamma_{i}},
$$
for some $l$, $\beta_{ij}$, $\gamma_{i}\in\mathbb{Z}$.
\end{lemma}

\begin{proof}
Let $R=\gamma_{3}(F_{n})F_{n}^{p^{m}}$ and write $w(x_{1},...,x_{n})=x_{i_{1}}^{p^{l_{1}}m_{1}}...x_{i_{k}}^{p^{l_{k}}m_{k}}$, where $i_{1}<i_{2}<...<i_{k}$, $l_{j}, m_{j}\in \mathbb{Z}$, $l_{j}\geq 0$ and the $m_{j}$ are non-zero and coprime to $p$ for all $j$. Choose $l_{t}$ minimal among the $l_{j}$ and let 
$$
y_{i_t}=x_{i_{1}}^{p^{l_{1}-l_{t}}m_{1}}...x_{i_{k}}^{p^{l_{k}-l_{t}}m_{k}}.
$$
Note that in the above expression for $y_{i_t}$ the exponent of $x_{i_t}$ is $m_{t}$. Then
$$
y_{i_t}^{p^{l_{t}}}\equiv x_{i_{1}}^{p^{l_{1}}m_{1}}...x_{i_{k}}^{p^{l_{k}}m_{k}}\displaystyle\prod_{i_{1}\leq i_{p}<i_{q}\leq i_{k}}[x_{i_{p}},x_{i_{q}}]^{-p^{l_{p}+l_{q}-l_{t}}m_{p}m_{q}\frac{p^{l_{t}}-1}{2}} \hspace{5 mm}\mbox{mod }\gamma_{3}(F_{n}),
$$
so that
\begin{eqnarray}
w(x_{1},...,x_{n})\equiv y_{i_t}^{p^{l_{t}}}\displaystyle\prod_{i<j}[x_{i},x_{j}]^{\beta_{i j}}\hspace{5 mm}\mbox{mod }\gamma_{3}(F_{n}),
\end{eqnarray}
for some $\beta_{i j}\in\mathbb{Z}$. Note that $x_{i_{t}}^{m_{t}}=(\prod_{s<t}x_{i_{s}}^{p^{l_{s}-l_{t}}m_{s}} )^{-1}y_{i_t}(\prod_{s>t}x_{i_{s}}^{p^{l_{s}-l_{t}}m_{s}})^{-1}$ and that since $p$ does not divide $m_{t}$ there exists a positive integer $r_{t}$ such that $x_{i_{t}}^{m_{t}r_{t}}\equiv x_{i_{t}}\hspace{1 mm}\mbox{mod }F_{n}^{p^m}$. Substituting the resulting expression for $x_{i_{t}}\hspace{0.5 mm}\mbox{mod }F_{n}^{p^m}$ into $(2)$ we have
$$
w(x_{1},...x_{n})\equiv y_{i_t}^{p^{l_{t}}}\displaystyle\prod_{\underset{i,j\neq i_{t}}{i<j}}[x_{i},x_{j}]^{\beta'_{i j}}\prod_{\underset{i\neq i_{t}}{i=1}}^{n}[y_{i_t},x_{i}]^{\gamma_{i}}\hspace{5 mm}\mbox{mod }R,
$$
for some $\beta'_{i j}, \gamma_{i}\in\mathbb{Z}$. The result follows in view of remark \ref{remark2}.
\end{proof}

\begin{lemma}
\label{LemmaVariableCommutator}
Let $w(x_{1},...,x_{n})=\prod_{i=1}^{n}\prod_{i<j}[x_{i},x_{j}]^{\alpha_{ij}}$ be an element of the free group of rank $n$ and $m\in\mathbb{N}$. 
Then, for any $G\in\mathfrak{N}_{2,m}$, $w$ is \textit{$G$-equivalent} to the word
$$
v(y_{1},...,y_{2k+1},x_{2k+2},...,x_{n})=\displaystyle\prod_{i=1}^{k}[y_{2i-1},y_{2i}]^{\gamma_{2i-1}\hspace{0.5mm}\gamma_{2i}}
$$
for some $\gamma_{ij}\in\mathbb{Z}$, where $2k\leq n$.
\end{lemma} 

\begin{proof}

Let $R=\gamma_{3}(F_n)F_{n}^{p^m}$. We are going to describe an algorithm which shows us that we may write $w$ in the form
$$
w(x_{1},...,x_{n})=\displaystyle\prod_{i=1}^{k}[y_{2i-1},y_{2i}]^{\gamma_{2i-1\hspace{0.5mm}2i}}c
$$
for some $\gamma_{ij}\in\mathbb{Z}$, where $2k\leq n$, the $y_{i}$ are words in the $x_i$s and $c\in \gamma_{3}(F_n)F_{n}^{p^m}$. In particular, each $y_{i}$ is of the  form $x_{i_1}^{p^{l_{i_1}}m_{i_1}}...x_{i_d}^{p^{l_{i_d}}m_{i_d}}$ the $m_{i_j}$ being non-zero and coprime to $p$ with $l_{i_{j}}\geq 0$ and $l_{i_{u_{i}}}=0$ for some $u_{i}$. Moreover, for all $y_i$, $y_j$ with $i\neq j$ we have $i_{u_{i}}\neq j_{u_{j}}$, i.e. $x_{i_{u_{i}}}\neq x_{j_{u_{j}}}$. The result then follows in view of remark \ref{remark2}. We proceed as follows:\\\\
Choose the first non zero $\alpha_{i j}$, with respect to the ordering in the product, say $\alpha_{s_{1} s_{11}}$. Then
\begin{eqnarray*}
w(x_{1},...,x_{n}) &=& \displaystyle\prod_{i<j}[x_{i},x_{j}]^{\alpha_{i j}}\\
                      &=& [x_{s_{1}},x_{s_{11}}^{\alpha_{s_{1} s_{11}}}...x_{s_{1q}}^{\alpha_{s_{1} s_{1{q}}}}]\displaystyle\prod_{s_{1}<i<j}[x_{i},x_{j}]^{\alpha_{i j}},
\end{eqnarray*}
with $s_{1}<s_{11}<s_{12}<...<s_{1q}$ and $\alpha_{s_{1} s_{1j}}$ non-zero for all $j$. Now $\alpha_{s_{1} s_{1j}}=p^{l_{s_{1} s_{1j}}}m_{s_{1} s_{1j}}$ where $l_{s_{1} s_{1j}}\geq 0$ is an integer and $m_{s_{1} s_{1j}}$ is coprime to $p$. So choose $l_{i j}$ minimal among the $l_{s_{1} s_{1j}}$, say $l_{s_{1} s_{1u_{1}}}$, and let
$$
y_{s_{1u_{1}}}=x_{s_{11}}^{p^{l_{s_{1} s_{11}}-l_{s_{1} s_{1u_{1}}}}m_{s_{1} s_{11}}}...x_{s_{1q}}^{p^{l_{s_{1} s_{1q}}-l_{s_{1} s_{1u_{1}}}}m_{s_{1} s_{1q}}}.
$$
Similarly to the previous lemma, 
$$
[x_{s_1},x_{s_{11}}^{\alpha_{s_1 s_{11}}}...x_{s_{1q}}^{\alpha_{s_1 s_{1q}}}]
\equiv[x_{s_1},y_{s_{1u_{1}}}]^{p^{l_{s_1 s_{1u_{1}}}}}\hspace{5 mm}\mbox{mod }\gamma_{3}(F_n).
$$ 
So
\begin{eqnarray}
w(x_{1},...,x_{n})\equiv [x_{s_1},y_{s_{1u_{1}}}]^{p^{l_{s_1 s_{1u_{1}}}}}\displaystyle\prod_{s_1<i<j}[x_{i},x_{j}]^{\alpha_{i j}}\hspace{5 mm}\mbox{mod }\gamma_{3}(F_n).
\end{eqnarray}
If the remaining $\alpha_{ij}$ are all zero we can stop here noting that in the expression for $y_{s_{1u_{1}}}$ the exponent of $x_{s_{1u_{1}}}$ is non-zero and coprime to $p$ and $x_{s_{1u_{1}}}\neq x_{s_1}$. Otherwise, as in the previous lemma, we may substitute $x_{s_{1u_{1}}}\hspace{0.5 mm}\mbox{mod }F_{n}^{p^m}$ from $\prod_{s_1<i<j}[x_{i},x_{j}]^{\alpha_{i j}}$ in $(3)$ giving us
$$
w\equiv [x_{s_1},y_{s_{1u_{1}}}]^{p^{l_{s_{1} s_{1u_{1}}}}}[y_{s_{1u_{1}}},x_{s_{21}}^{\alpha'_{s_{1u_{1}}s_{21}}}...x_{s_{2r}}^{\alpha'_{s_{1u_{1}} s_{2r}}}]\displaystyle\underset{i,j\neq s_{1u_{1}}}{\prod_{s_{1}<i<j}}[x_{i},x_{j}]^{\alpha'_{i j}}\hspace{5 mm}\mbox{mod }R,
$$
where $\alpha'_{i j}\in\mathbb{Z}$ and $s_{1}<s_{21}<...<s_{2r}$. Note that $s_{2j}\neq s_{1u_{1}}$ for all $j$. Now, $\alpha'_{s_{1u_{1}} s_{2i}}=p^{l'_{s_{1u_{1}} s_{2i}}}m'_{s_{1u_{1}} s_{2i}}$ with the $m'_{s_{1u_{1}} s_{2i}}$ non-zero and coprime to $p$. Again choose $l'_{s_{1u_{1}} s_{2u_{2}}}$ minimal among the $l'_{s_{1u_{1}} s_{2i}}$ and let
$$
y_{s_{2u_{2}}}=x_{s_{21}}^{p^{l'_{s_{1u_{1}} s_{21}}-l'_{s_{1u_{1}} s_{2u_{2}}}}m'_{s_{1u_{1}} s_{21}}}...x_{s_{2r}}^{p^{l'_{s_{1u_{1}} s_{2r}}-l'_{s_{1u_{1}} s_{2u_{2}}}}m'_{s_{1u_{1}} s_{2r}}}.
$$
Note that the exponent of $x_{s_{2u_{2}}}$ in the expression for $y_{s_{2u_{2}}}$ is non-zero and coprime to $p$ and that $x_{s_{2u_{2}}}\neq x_{s_{1u_{1}}}$. Then, similarly to before,
\begin{eqnarray*}
[y_{s_{1u_{1}}},x_{s_{21}}^{\alpha'_{s_{1u_{1}} s_{21}}}...x_{s_{2r}}^{\alpha'_{s_{1u_{1}} s_{2r}}}] 
                        &\equiv& [y_{s_{1u_{1}}},y_{s_{2u_{2}}}]^{p^{l'_{s_{1u_{1}} s_{2u_{2}}}}}\hspace{5 mm}\mbox{mod }\gamma_{3}(F_n).
\end{eqnarray*}
Thus
\begin{eqnarray}
w=[x_{s_1},y_{s_{1u_{1}}}]^{p^{l_{s_1 s_{1u_{1}}}}}[y_{s_{1u_{1}}},y_{s_{2u_{2}}}]^{p^{l'_{s_{1u_{1}} s_{2u_{2}}}}}\displaystyle\prod_{\underset{i,j\neq s_{1u_{1}}}{s_{1}<i<j}}[x_{i},x_{j}]^{\alpha'_{ij}}\hspace{5 mm}\mbox{mod }R.
\end{eqnarray}
There are two cases to consider.
\newcounter{counter2}
\begin{list}{\textit{Case (\arabic{counter2}a):}}{\usecounter{counter2}}
\item The first case is when $l_{s_1 s_{1u_{1}}}\leq l'_{s_{1u_{1}} s_{2u_{2}}}$. We have
\begin{eqnarray*}
[x_{s_1},y_{s_{1u_{1}}}]^{p^{l_{s_1 s_{1u_{1}}}}}[y_{s_{1u_{1}}},y_{s_{2u_{2}}}]^{p^{l'_{s_{1u_{1}} s_{2u_{2}}}}}\equiv [y_{s_{1u_{1}}},y_{s_{1}}]^{p^{l_{s_1 s_{1u_{1}}}}}\hspace{5 mm}\mbox{mod }\gamma_{3}(F_n),
\end{eqnarray*}
where $y_{s_{1}}=x_{s_1}^{-1}y_{s_{2u_{2}}}^{p^{l'_{s_{1u_{1}} s_{2u_{2}}}-l_{s_1 s_{1u_{1}}}}}$. This gives
\begin{eqnarray}
w(x_{1},...,x_{n})\equiv[y_{s_{1u_{1}}},y_{s_{1}}]^{p^{l_{s_{1} s_{1u_{1}}}}}\displaystyle\prod_{\underset{i,j\neq s_{1u_{1}}}{s_{1}<i<j}}[x_{i},x_{j}]^{\alpha'_{ij}}\hspace{5 mm}\mbox{mod }R.
\end{eqnarray}
Note that in the expression for $y_{s_{1}}$ the exponent of $x_{s_{1}}$ is non-zero and coprime to $p$ as is the exponent of $x_{s_{1u_{1}}}$ in $y_{s_{1u_{1}}}$. In particular note that neither $x_{s_{1}}$ nor $x_{s_{1u_{1}}}$ appear in the rest of the expression for $w\mbox{ mod }R$ and that $x_{s_{1u_{1}}}\neq x_{s_{1}}$. We can set $y_{1}:=y_{s_{1u_{1}}}$ and $y_{2}:=y_{s_{1}}$ and repeat the algorithm for the rest of the commutators.
\item
The second case however is when $l_{s_{1} s_{1u_{1}}}>l'_{s_{1u_{1}} s_{2u_{2}}}$. We have
\begin{eqnarray*}
[x_{s_1},y_{s_{1u_{1}}}]^{p^{l_{s_{1} s_{1u_{1}}}}}[y_{s_{1u_{1}}},y_{s_{2u_{2}}}]^{p^{l'_{s_{1u_{1}} s_{2u_{2}}}}}\equiv[y_{s_{1u_{1}}},y_{t_{2u_{2}}}]^{p^{l'_{s_{1u_{1}} s_{2u_{2}}}}}\hspace{5 mm}\mbox{mod }\gamma_{3}(F_n)
\end{eqnarray*}
where $y_{t_{2u_{2}}}=x_{s_{1}}^{-p^{l_{s_{1} s_{1u_{1}}}-l'_{s_{1u_{1}} s_{2u_{2}}}}}y_{s_{2u_{2}}}$. This gives
\begin{eqnarray}
w(x_{1},...,x_{n})\equiv [y_{s_{1u_{1}}},y_{t_{2u_{2}}}]^{p^{l'_{s_{1u_{1}} s_{2u_{2}}}}}\displaystyle\prod_{\underset{i,j\neq s_{1u_{1}}}{s_{1}<i<j}}[x_{i},x_{j}]^{\alpha'_{ij}}\hspace{5 mm}\mbox{mod }R.
\end{eqnarray}
Here the exponent of $x_{s_{1u_{1}}}$ in $y_{s_{1u_{1}}}$ is non-zero and coprime to $p$ and the same is true for the exponent of $x_{s_{2u_{2}}}$ in $y_{t_{2u_{2}}}$. However, in contrast to Case ($1a$), whilst $x_{s_{1u_{1}}}$ does not appear in the rest of the expression for $w\mbox{ mod }R$ it is possible that $x_{s_{2u_{2}}}$ may. If it doesn't then set $y_{1}:=y_{s_{1u_{1}}}$ and $y_{2}:=y_{t_{2u_{2}}}$ and repeat the algorithm for the rest of the commutators, if there are any. If it does substitute it out of the expression in $(6)$ using 
\begin{eqnarray*}
y_{t_{2u_{2}}} &=& x_{s_{1}}^{-p^{l_{s_{1} s_{1u_{1}}}-l'_{s_{1u_{1}} s_{2u_{2}}}}}y_{s_{2u_{2}}}\\ &=& x_{s_{1}}^{-p^{l_{s_{1} s_{1u_{1}}}-l'_{s_{1u_{1}} s_{2u_{2}}}}}x_{s_{21}}^{p^{l'_{s_{1u_{1}} s_{21}}-l'_{s_{1u_{1}} s_{2u_{2}}}}m'_{s_{1u_{1}} s_{21}}}...x_{s_{2r}}^{p^{l'_{s_{1u_{1}} s_{2r}}-l'_{s_{1u_{1}} s_{2u_{2}}}}m'_{s_{1u_{1}} s_{2r}}}.
\end{eqnarray*}
This gives
$$
w\equiv [y_{s_{1u_{1}}},y_{t_{2u_{2}}}]^{p^{l'_{s_{1u_{1}} s_{2u_{2}}}}}[y_{t_{2u_{2}}},x_{s_{31}}^{\alpha''_{t_{2u_{2}} s_{31}}}...x_{s_{3k}}^{\alpha''_{t_{2u_{2}} s_{3k}}}]\displaystyle\underset{i,j\neq s_{1u_{1}},s_{2u_{2}}}{\prod_{i<j}}[x_{i},x_{j}]^{\alpha''_{i j}}\hspace{5 mm}\mbox{mod }R,
$$
where, as usual, $\alpha''_{t_{2u_{2}} s_{3j}}=p^{l''_{t_{2u_{2}} s_{3j}}}m''_{t_{2u_{2}} s_{3j}}$ for all $j$. Now choose $l''_{t_{2u_{2}} s_{3u_{3}}}$ minimal among the $l''_{t_{2u_{2}} s_{3j}}$ and let
$$
y_{s_{3u_{3}}}=x_{s_{31}}^{p^{l''_{t_{2u_{2}} s_{31}}-l''_{t_{2u_{2}} s_{3u_{3}}}}m''_{t_{2u_{2}} s_{31}}}...x_{s_{3k}}^{p^{l''_{t_{2u_{2}} s_{3k}}-l''_{t_{2u_{2}} s_{3u_{3}}}}m''_{t_{2u_{2}} s_{3k}}}.
$$
Then, as before, the expression for $w$ becomes
\begin{eqnarray}
w\equiv [y_{s_{1u_{1}}},y_{t_{2u_{2}}}]^{p^{l'_{s_{1u_{1}} s_{2u_{2}}}}}[y_{t_{2u_{2}}},y_{s_{3u_{3}}}]^{p^{l''_{t_{2u_{2}} s_{3u_{3}}}}}\displaystyle\underset{i,j\neq s_{1u_{1}},s_{2u_{2}}}{\prod_{i<j}}[x_{i},x_{j}]^{\alpha''_{i j}}\hspace{5 mm}\mbox{mod }R.
\end{eqnarray}
\end{list}
Again there are two cases. 
\newcounter{counter3}
\begin{list}{\textit{Case (\arabic{counter3}b):}}{\usecounter{counter3}}
\item
If $ l'_{s_{1u_{1}} s_{2u_{2}}}\leq l''_{t_{2u_{2}} s_{3u_{3}}}$ then $w$ becomes
\begin{eqnarray}
w(x_{1},...,x_{n})\equiv [y_{t_{2u_{2}}},y_{v_{1{u_1}}}]^{p^{l'_{s_{1u_{1}} s_{2u_{2}}}}}\displaystyle\underset{i,j\neq s_{1u_{1}},s_{2u_{2}}}{\prod_{i<j}}[x_{i},x_{j}]^{\alpha''_{i j}}\hspace{5 mm}\mbox{mod }R,
\end{eqnarray}
where $y_{v_{1u_{1}}}=y_{s_{1u_{1}}}^{-1}y_{s_{3u_{3}}}^{p^{l''_{t_{2u_{2}} s_{3u_{3}}}-l'_{s_{1u_{1}} s_{2u_{2}}}}}$ and we are done as in Case ($1a$). 
\item If however $l'_{s_{1u_{1}}  s_{2u_{2}}}>l''_{t_{2u_{2}} s_{3u_{3}}}$ we can repeat the algorithm described in Case $(2a)$ above and we will end up with an expression of the form
\begin{eqnarray}
w\equiv [y_{t_{2u_{2}}},y_{t_{3u_{3}}}]^{p^{l''_{t_{2u_{2}} s_{3u_{3}}}}}[y_{t_{3u_{3}}},y_{s_{4u_{4}}}]^{p^{l'''_{t_{3u_{3}} s_{4u_{4}}}}}\displaystyle\underset{i,j\neq s_{2u_{2}},s_{3u_{3}}}{\prod_{i<j}}[x_{i},x_{j}]^{\alpha'''_{i j}}\hspace{5 mm}\mbox{mod }R,
\end{eqnarray}
where $y_{t_{3u_{3}}}$ and $y_{s_{4u_{4}}}$ are words in the $x_{i}$s with the exponents of $x_{s_{3u_{3}}}$ and $x_{s_{4u_{4}}}$ non-zero and coprime to $p$ respectively. Again there are two cases. If in $(9)$ we have $l''_{t_{2u_{2}}s_{3u_{3}}}\leq l'''_{t_{3u_{3}} s_{4u_{4}}}$ then, as in Cases ($1a$) and ($1b$), we are done. If not we have $l_{s_{1} s_{1u_{1}}}>l'_{s_{1u_{1}} s_{2u_{2}}}>l''_{t_{2u_{2}} s_{3u_{3}}}>l'''_{t_{3u_{3}} s_{4u_{4}}}\geq...\geq0$ and we keep going until the algorithm stops, i.e. we are in the first case.
\end{list}
We will eventually end up with an expression like $(5)$ or $(8)$. If the remaining $\alpha''_{i j}$, in $(8)$ say, are all zero we can stop here. Otherwise, set $y_{1}:=y_{t_{2u_{2}}}$ and $y_{2}:=y_{v_{1u_{1}}}$ and repeat the algorithm on the rest of the commutators. Eventually, after relabelling, we have an expression for $w$ of the desired form.
\end{proof}
We are now ready to prove the theorem. Let $w(x_{1},...,x_{n})$ be any group word and fix a group $G\in\mathfrak{N}_{2,m}$.  From $(1)$ we have
\begin{eqnarray*}
w(x_{1},...,x_{n})\equiv x_{1}^{\alpha_{1}}...x_{n}^{\alpha_{n}}(\displaystyle\prod_{i<j}[x_{i},x_{j}]^{\beta_{i j}})\hspace{5 mm}\mbox{mod }\gamma_{3}(F_n).
\end{eqnarray*}
By Lemma \ref{LemmaVariablePower} $w$ is \textit{$G$-equivalent} to the word
\begin{eqnarray*}
w'(y_{1},x_{2},...,x_{n})=y_{1}^{p^{l}}\displaystyle\prod_{i<j}[x_{i},x_{j}]^{\beta'_{i j}}
[y_{1},h]
\end{eqnarray*}
for some $l$, $\beta'_{ij}\in\mathbb{Z}$, where $h$ is some word in the $x_{i}$s with $i\neq 1$. As in the proof of Lemma \ref{LemmaVariablePower} we can substitute $x_{1}$ out of the expression above giving us
\begin{eqnarray*}
w'(y_{1},x_{2},...,x_{n})\equiv y_{1}^{p^{l}}\displaystyle\prod_{1<i<j}[x_{i},x_{j}]^{\beta''_{i j}}
[y_{1},h']
\hspace{5 mm}\mbox{mod }R
\end{eqnarray*}
for some $\beta''_{ij}\in\mathbb{Z}$, where $R$ denotes $\gamma_{3}(F_{n})F_{n}^{p^{m}}$ and $h'$ is some word in the $x_{i}$s with $i\neq1$. 
Then by Lemma \ref{LemmaVariableCommutator} $w'$ is \textit{$G$-equivalent} to the word
$$
v(y_{1},y_{2},...,y_{2k+1},x_{2k+2},...,x_{n})=y_{1}^{p^{l}}\displaystyle\prod_{i=1}^{k}[y_{2i},y_{2i+1}]^{\gamma_{2i\hspace{0.5 mm}2i+1}}
[y_{1},h'']
$$
for some $\gamma_{ij}\in\mathbb{Z}$, where $2k+1\leq n$ and $h''$ is some word in the $y_{i}$s and $x_{j}$s for $i=2,...,2k+1$ and $j=2k+2,...,n$. We write this as
$$
v(y_{1},y_{2},z_{2},...,y_{k+1},z_{k+1},x_{2k+2},...,x_{n})=y_{1}^{p^{l}}\displaystyle\prod_{i=2}^{k+1}[y_{i},z_{i}]^{\upsilon_i}
[y_{1},h''],
$$
where $\upsilon_{i+1}=\gamma_{2i\hspace{0.5 mm}2i+1}$. Consider the word map given by $v$. The commutator map from $G\times G$ to $G$ sending a pair $(x,y)$ to its commutator $[x,y]$ is a bilinear map. Fixing the $z_{i}$ for all $i$ and restricting $y_{1}$ to the derived group of $G$ we obtain a linear map
$$
v': G'\times G^{(n-1-k)}\rightarrow G',
$$
defined by $v'(y_{1},...,y_{k+1},x_{2k+2},...,x_{n})=v(y_{1},y_{2},z_{2},...,y_{k+1},z_{k+1},x_{2k+2},...,x_{n})$.
Now $|\{\textbf{g}\in G'\times G^{(n-1-k)} : v'(\textbf{g})=1\}|\geq|G|^{n-1-k}$ and so $N(G,v)\geq|G|^{n-1}$ since we had $|G|$ choices for each of the $z_{i}$. Thus $P(G,v)\geq\frac{1}{|G|}$ and the result follows since $P(G,w)=P(G,v)$.



\end{document}